\documentclass[11pt]{llncs}
\usepackage{amsmath,amssymb,mathtools}
\usepackage[all]{xy}
\usepackage{fullpage}
\DeclareMathAlphabet{\mathpzc}{OT1}{pzc}{m}{it}
\mathchardef\mhyphen="2D
\begin{document}
%
%\frontmatter          % for the preliminaries
%
\pagestyle{headings}  % switches on printing of running heads
%\addtocmark{Hamiltonian Mechanics} % additional mark in the TOC
%
%\mainmatter              % start of the contributions
%
\title{Pairings from a tensor product point of view}
%\title{A tensor approach to pairings}
%
\titlerunning{Pairings from a tensor product point of view}  % abbreviated title (for running head)
%                                     also used for the TOC unless
%                                     \toctitle is used
%
\author{Nadia El Mrabet\inst{1}\thanks{One of the authors, Nadia El Mrabet, wishes to acknowledge support from French project ANR INS 2012 SYMPATIC.} and Laurent Poinsot\inst{2}}
\authorrunning{Nadia El Mrabet and Laurent Poinsot} % abbreviated author list (for running head)
%
%%%% list of authors for the TOC (use if author list has to be modified)
%\tocauthor{Ivar Ekeland, Roger Temam, Jeffrey Dean, David Grove,
%Craig Chambers, Kim B. Bruce, and Elisa Bertino}
%
\institute{Universit\'e Paris 8, LIASD, France,\\
\email{elmrabet@ai.univ-paris8.fr},\\ \texttt{http://www.ai.univ-paris8.fr/\~{}elmrabet/}
\and
Universit\'e Paris 13, Sorbonne Paris Cit\'e, LIPN, CNRS (UMR 7030), France,\\
\email{laurent.poinsot@lipn.univ-paris13.fr},\\ 
\texttt{http://lipn.univ-paris13.fr/\~{}poinsot/}
}

\maketitle              % typeset the title of the contribution

\begin{abstract}
Pairings are particular bilinear maps, and as any bilinear maps they factor through the tensor product as group homomorphisms. Besides, nothing seems to prevent us to construct pairings on other abelian groups than elliptic curves or more general abelian varieties. The point of view adopted in this contribution is based on these two observations. Thus we present an elliptic curve free study of pairings which is essentially based on tensor products of abelian groups (or modules). Tensor products of abelian groups are even explicitly computed under finiteness conditions. We reveal that the existence of pairings depends on the non-degeneracy of some universal bilinear map, called the canonical bilinear map. In particular it is shown that the construction of a pairing on $A\times A$ is always possible whatever a finite abelian group $A$ is. We also propose some new constructions of pairings, one of them being based on the notion of group duality which is related to the concept of non-degeneracy. 
\keywords{Pairing, tensor product, finite abelian group, module, duality.}\\
\quad\\
\textbf{Mathematics Subject Classification (2010) 15A69, 11E39}
\end{abstract}

\section{Introduction}

A bilinear map is a function of two variables that belong to two finite abelian groups, and with values in another abelian group, such that when fixing one of its variable the map thus obtained is a homomorphism of groups. Bilinear maps were originally introduced in cryptography in order to solve the discrete logarithm problem~\cite{MOV1993}. Due to bilinearity it is possible to transport this problem from a group for which it is assumed to be difficult to another one where the problem becomes easier. Afterwards, bilinear maps were used to define tripartite Diffie-Hellman key exchange protocol~\cite{Joux}. In these two situations, the bilinear maps under consideration are assumed to be non-degenerate, and are called pairings. For such a map $f\colon A\times B\rightarrow C$, this means that apart from the identity element of $A$ (respectively, $B$),  there is no members of $A$ (respectively, of $B$) that annihilate every member of $B$ (respectively, of $A$). For these kinds of use the groups $A,B,C$ are cyclic groups. Many pairings considered in the literature are naturally associated to some objects arising in algebraic (projective) geometry such as elliptic curves and more generally abelian varieties. For a long time pairings were variants of the Weil~\cite{Silverman} and Tate~\cite{Ruck} pairings over genus $1$ or $2$ curves over finite fields. More recently pairings over more general abelian varieties have been proposed~\cite{Lubicz} and even based on dot-products~\cite{Okamoto1} for homomorphic encryption. 

More attention was given to pairings over elliptic curves for at least two reasons. First of all, it seems that the security level of such pairings with respect to the discrete logarithm problem and to pairing inversions is high (see for instance~\cite{Boxall}). Secondly, these pairings may be computed rather efficiently (with help of an efficient finite field arithmetic~\cite{Bajard,ElMrabet} or by optimized versions of Miller's algorithm~\cite{Hess,Vercauteren}).  Apart from these two important cryptographic issues, pairings are bilinear maps between finite abelian groups, and as any bilinear map, a pairing descends to the tensor product of abelian groups as a  usual group homomorphism. It seems rather natural to study pairings through the notion of tensor product, and it is the point of view adopted in this contribution. More precisely, we study, and construct, bilinear maps between finite abelian groups (and more generally between modules over some fixed ring) seen as homomorphisms from a tensor product to an abelian group (or a module). This provides an elliptic curve free presentation of pairings between abstract groups. Not all our results are difficult, some of them are folklore, and lot of them may even be  qualified as simple for group-theorists, but we think that one of the main worth of this work is to provide a unified treatment of pairings in an abstract setting. This approach is quite natural since many properties of pairings are independent from algebraic geometry. Because we have chosen to work at the abstract groups level, we do not deal with the cryptographic issues of efficiency and security. We believe that these gaps are balanced by the results stated in this contribution, and our rather general approach to pairings. We also believe that this work may serve as a basis for new constructions of cryptographically relevant pairings on other group structures than elliptic curves (see for instance~\cite{Lubicz,Okamoto1}). 

The remainder of this contribution is organized as follows: Section~\ref{basic} fixes the general notations, provides basic definitions about bilinear maps and pairings, and contains a brief overview on pairing-based cryptography. Section~\ref{tensorprod} is about the tensor product of groups and modules themselves, of which it provides a number of useful properties. Section~\ref{finiteAbgroups} is entirely devoted to the tensor product of finite abelian groups: the rules to compute any such tensor product are presented. It also deals with the canonical bilinear map (which is canonically attached to a tensor product) and the fact that non-degeneracy of a bilinear map depends of that of a canonical bilinear map.  Section~\ref{constructions} contains several constructions of pairings. Some properties about known pairings are also recovered. 

\section{An introduction to pairing-based cryptography}\label{basic}

\subsection{Some notations and definitions}

Before introducing the notion of pairings and their use in cryptography, let us begin with some notations, useful hereafter in this contribution.

Let $f\colon X\times Y\rightarrow Z$ be any set-theoretic map. For any $x\in X$, we define the map $f(x,\cdot)\colon X\rightarrow Z$ by $y\mapsto f(x,y)$, and symmetrically, for any $y\in Y$ is defined $f(\cdot,y)\colon Y\rightarrow Z$ by $x\mapsto f(x,y)$.  
The identity element of a group $G$ is denoted either by $1_G$ or by $0_G$ (or $1$ or $0$) whether $G$ is given in multiplicative or additive notation. Let $G,H,K$ be three groups (abelian or not). A map $f\colon G\times H\rightarrow K$ is said to be \emph{bilinear} if for every $g\in G$, and every $h\in H$, the maps $f(g,\cdot)\colon H\rightarrow K$ and $f(\cdot,h)\colon G\rightarrow K$ are homomorphisms of groups. The set of all bilinear maps from $G\times H$ to $K$ is then denoted by $\mathpzc{Bil}(G\times H,K)$. Actually, this notion may be defined in another setting, that of modules over some commutative ring. In this contribution, $R$ always denotes a commutative ring with a unit $1_R$. Given three $R$-modules, $A,B,C$, we say that a map $f\colon A\times B\rightarrow C$ is \emph{$R$-bilinear} whenever for every $a\in A$ and every $b\in B$, the maps $f(a,\cdot)\colon B\rightarrow C$ and $f(\cdot,b)\colon A\rightarrow C$ are $R$-linear. When $R=\mathbb{Z}$, then $\mathbb{Z}$-bilinear maps are exactly bilinear maps between abelian groups. In what follows, the set of all $R$-bilinear maps from $A\times B$ to $C$ is denoted by $\mathpzc{Bil}_R(A\times B,C)$. Moreover, we denote $\mathpzc{Bil}_{\mathbb{Z}}(A\times B,C)$ simply by $\mathpzc{Bil}(A\times B,C)$ since when $A,B,C$ are abelian groups both notions of bilinearity coincide. Continuing with notations, if $G,H$ are groups, then $\mathpzc{Hom}(G,H)$ is the set of all group homomorphisms from $G$ to $H$, and if $A,B$ are two $R$-modules, then $\mathpzc{Hom}_R(A,B)$ is the set of all $R$-linear maps from $A$ to $B$. Again if $A,B$ are abelian groups, then $\mathpzc{Hom}(A,B)=\mathpzc{Hom}_{\mathbb{Z}}(A,B)$.
\begin{example}
Let $A$ be an abelian group. Let $R^{\star}$ be the group of invertible elements of $R$. A bilinear map $f\colon A\times A\rightarrow R^{\star}$  is called a \emph{bicharacter}~\cite{Ree}. When furthermore $f(a,b)f(b,a)=1$ and $f(a,a)=\pm 1$ for every $a,b\in A$, $f$ is said to be a \emph{commutation factor}~\cite{Scheunert}. Such commutation factors are used to define color Lie superalgebras~\cite{Bah}. 
\end{example}

One of the main feature of a pairing (the definition of which is recalled hereafter) is the notion of non-degeneracy. Let $G,H,K$ be three groups and $A,B,C$ be three $R$-modules. Let $f\in \mathpzc{Bil}(G\times H,K)$ (respectively, $f\in \mathpzc{Bil}_R(A\times B,C)$). The map $f$ is said to be \emph{left non-degenerate} if the map $g\in G\mapsto f(g,\cdot)$ (respectively, $a\in A\mapsto f(a,\cdot)$) is one-to-one. In other terms this means that  if for every $h\in H$ (respectively, every $b\in B$), $f(g,h)=1_K$ (respectively, $f(a,b)=0_C$), then $g=1_G$ (respectively, $a=0_A$). The notions of \emph{right non-degeneracy} are the evident symmetric ones, while we say that a bilinear map $f$ is \emph{non-degenerate} whenever it is both left and right non-degenerate. The map $f$ is said to be \emph{degenerate} if it is not non-degenerate. In its original form, a \emph{pairing} is a non-degenerate bilinear map between finite abelian groups. For our purpose the definition of a pairing is somewhat extended to allow pairings between non-abelian groups or $R$-modules. In brief, a \emph{pairing} is a non-degenerate map  $f\in \mathpzc{Bil}(G\times H,K)$ (respectively, $f\in\mathpzc{Bil}_R(A\times B,C)$) where $G,H,K$ are groups (respectively, $A,B,C$ are $R$-modules). In particular, there is no size issue in the definition of a pairing although our examples will be given under finiteness assumptions. %A bilinear map $f\colon A\times B\rightarrow C$ is said to be \emph{perfect} if $a\in A \mapsto f(a,\cdot)\in \mathpzc{Hom}_{R}(B,C)$ and $b\in B\mapsto f(\cdot,b)\in\mathpzc{Hom}_{R}(A,C)$ are isomorphisms. Obviously a perfect bilinear map is a pairing and we refer to it as a \emph{perfect pairing}. 
We also sometimes use the traditional  ``bracket'' notation $\langle\cdot\mid\cdot\rangle$ to denote a pairing. 
\begin{example}\label{ex-abelian-central-quotient}
Let $1\rightarrow A\rightarrow G \rightarrow B\rightarrow 1$ be a short exact sequence of groups, where $A, B$ are abelian groups, and $A$ lies in the center $Z(G)$ of $G$ (\emph{i.e.}, $G$ is a central extension of abelian groups). Let $[g,h]=ghg^{-1}h^{-1}$ be the \emph{commutator} of $g,h\in G$.  According to~\cite{Baer}, $[\cdot,\cdot]$ descends to the quotient as a bilinear map $[\cdot,\cdot]\colon B\times B\rightarrow A$. Moreover it is alternating (\emph{i.e.}, $[x,x]=1$ for every $x\in B$). Finally, it is non-degenerate if, and only if, $A=Z(G)$, so that we obtain a pairing $[\cdot,\cdot]\colon G/Z(G)\times G/Z(G)\rightarrow Z(G)$ (whenever $G/Z(G)$ is abelian).
\end{example}

\subsection{Background on pairing-based cryptography}

We recall here the basic facts and definitions of pairings over elliptic curves.   Let $r$ be a prime integer, $A,B,C$  be three abelian groups of order $p$. A pairing is a bilinear and non-degenerate map $f \colon  A \times B \rightarrow C$. We briefly present the most frequent choices for $A$, $B$ and $C$ in pairing-based cryptography. 
Let $E$ be an elliptic curve over the finite field $\mathbb{F}_q$ of characteristic $p$. The integer $r$ is chosen to be a prime divisor of $|E(\mathbb{F}_q)|$, co-prime with $p$. A pairing is usually defined over the points of $r$-torsion of $E$: $E[r] = \{\, P \in E(\overline{\mathbb{F}_q})\colon rP=P_{\infty} \,\}$, where $P_{\infty}$ is the point at infinity of the elliptic curve. We know that $ E[r] \cong \mathbb{Z}/{r\mathbb{Z}} \times \mathbb{Z}/{r\mathbb{Z}}$ \cite[Chap III Cor. 6.4]{Silverman}. The embedding degree $k$ of $E$ relatively to $r$ is the smallest integer such that $r$ divides $(q^k-1)$. A result of Balasubramanian and Koblitz~\cite{BSS2005} ensures that, when $k>1$, all the points of $E[r]$ are rational over the extension $\mathbb{F}_{q^k}$ of degree $k$, \emph{i.e.}, $E[r]=E(\mathbb{F}_{q^k})$. The group $A$ is then the subgroup generated by a point $P \in E(\mathbb{F}_q)$ of order $r$. The subgroup $B$ is chosen as another subgroup of order $r$ of $E[r]$, a popular choice is the subgroup generated by a point $Q$ of order $r$ over $E(\mathbb{F}_{q^k})$, such that $\Pi(Q)=qQ$, where $\Pi$ represents the Frobenius endormorphism over $\mathbb{F}_q$. Finally, the group $C$ is the unique subgroup of order $r$ of $\mathbb{F}_{q^k}^{*}$ (it exists and is unique because $r$ divides $(q^k-1)$ and $\mathbb{F}_{q^k}^{*}$ is a cyclic group).  
This choice of subgroups may be seen as the restriction to $A \times B$ of the Weil pairing on $E[r] \times E[r]$, or the Tate paring, or one of its variant (reduced Tate, Ate, twisted Ate, optimal pairing or pairing lattices). The Miller algorithm is used to computed all these pairings.

The original objective of pairings in cryptography was to solve the discrete logarithm problem. The pairings shift the discrete logarithm problem from a subgroup over an elliptic curve to a discrete logarithm problem  over a finite field. The interest is that the discrete logarithm problem is easier on finite fields compared to elliptic curves~\cite{MOV1993}. Later, the pairings were used to compose the tripartite Diffie-Hellman key exchange~\cite{Joux}. It was a simplification of the Diffie-Hellman key construction between three entities. Nowadays, pairings are used for several protocols such as identity based cryptography~\cite{BonehF01} or short signature schemes~\cite{JONE2008}. The security of pairing-based cryptography lays on the discrete logarithm problem over the three groups $A$, $B$ and $C$~\cite{Boxall}.

\section{Tensor product of groups (and modules)}\label{tensorprod}

The notions of bilinear maps and tensor product are closely related as it is explained hereafter, and this relation is exploited in section~\ref{constructions} to construct new pairings on finite abelian groups. In brief, every bilinear map factors through a quotient group -- the tensor product -- as a linear map. The original bilinear map is recovered by composing this linear map with a  ``universal'' bilinear map. Therefore the study of bilinear maps reduces to that of a unique (and universal) bilinear map and of those linear maps which are defined on a particular kind of groups (or $R$-modules), namely the tensor product. In this section are recalled the constructions of the tensor product of groups and modules together with some of their basic properties. We also explain the reason why it is somewhat useless to define bilinear maps (or pairings) on non-abelian groups. Other properties of bilinear maps in the setting of finite abelian groups are presented in section~\ref{finiteAbgroups}.   

\subsection{Free (commutative) group and abelianization}

The basic notions recalled in this subsection may be found for instance in~\cite{BouAlg}. 

Let $G$ be a group. For any elements $g,h\in G$, the \emph{commutator} of $g$ and $h$ is $[g,h]=g^{-1}h^{-1}gh$ (see example~\ref{ex-abelian-central-quotient}). The \emph{derived subgroup} $[G,G]$ is generated by all the commutators and it turns to be a normal subgroup of $G$. It is even the smallest normal subgroup such that the quotient group of $G$ by this subgroup is abelian.Thus the quotient group $G/[G,G]$, denoted by $\mathpzc{Ab}(G)$, is an abelian group, called the \emph{abelianization of $G$}. It satisfies the following property: let $A$ be an abelian group, and $f\colon G\rightarrow A$ be a homomorphism of groups, then there is a unique homomorphism of groups $g\colon \mathpzc{Ab}(G)\rightarrow A$ such that $g\circ\pi=f$, where $\pi$ denotes the natural epimorphism $G\rightarrow \mathpzc{Ab}(G)$. 

Let $X$ be a set. There exists a way to construct a group $F(X)$, called the \emph{free group over  $X$}, that contains $X$,  and which is the solution\footnote{Actually a solution of a  universal problem  is only unique up to a unique isomorphism in some category, see~\cite{MacLane}.} of the following ``universal problem'': for any group $G$ and any set-theoretic map $f\colon X\rightarrow G$, there exists a unique homomorphism of groups $g\colon F(X)\rightarrow G$ such that $g(x)=f(x)$ for every $x\in X$. The construction is made as follows: for each $x\in X$, we introduce a new symbol, say $\overline{x}$, and we let $\overline{X}$ denote the totality of these symbols. Then, we consider the free monoid $(X\cup \overline{X})^*$ over $X\cup \overline{X}$. It consists of all  words (including the empty word $\epsilon$), \emph{i.e.},  finite sequences of elements of $X\cup \overline{X}$. The composition of words is the obvious one (concatenation), and $\epsilon$ acts as the identity. Finally, let $\cong$ be the least congruence of $(X\cup \overline{X})^*$ containing $\{\, (x\overline{x},\epsilon)\colon x\in X\,\}\cup\{\, (\overline{x}x,\epsilon)\colon x\in X\,\}$ (see~\cite{Cliff}). It turns that the quotient monoid $(X\cup \overline{X})^*/\cong$ (also known as the \emph{Grothendieck group completion} of $(X\cup  \overline{X})^*$) is actually a group which is precisely the free group $F(X)$. 

There also exists a similar construction for abelian groups, and more generally for modules. Recall that  $R$ is a commutative ring with a unit $1_R$. An element $f\in R^X$ is said to be \emph{finitely supported} whenever the set of all $x\in X$ such that $f(x)\not=0$ is finite.  For instance, for each $x\in X$, the map $\delta_x\in R^X$ that vanishes at all $y\not=x$, and such that $\delta_x(x)=1_R$ is finitely supported. The set of all such functions is denoted by $R^{(X)}$. It is a free $R$-module with basis $X$ (under identification of $\delta_x$ with $x$ for each $x\in X$). In particular, for $R=\mathbb{Z}$, we obtain the \emph{free commutative group} $\mathbb{Z}^{(X)}$ on $X$. As it is expected, $\mathpzc{Ab}(F(X))\cong \mathbb{Z}^{(X)}$ (isomorphic as groups). 

\subsection{Tensor product: construction and properties}

We are now in position to introduce the tensor product of groups and $R$-modules. 
Let $G,H$ be two groups (in multiplicative notation), and let $N$ be the normal subgroup of $F(G\times H)$ generated by the elements $(gg^{\prime},h)(g,h)^{-1}(g^{\prime},h)^{-1}$ and $(g,hh^{\prime})(g,h)^{-1}(g,h^{\prime})^{-1}$ for all $g,g^{\prime}\in G$, $h,h^{\prime}\in H$. The quotient group $F(G\times H)/N$ is denoted by $G\otimes H$. We denote by $g\otimes h$ the image of $(g,h)\in G\times H$ in $G\otimes H$ and this clearly defines a bilinear map from $G\times H$ to $G\otimes H$ called the \emph{canonical bilinear map}. The group $G\otimes H$ also satisfies a universal property: for every group $K$ and every bilinear map $f\colon G\times H\rightarrow K$, there exists a unique homomorphism of groups $f^{\prime}\colon G\otimes H\rightarrow K$ such that $f^{\prime}(g\otimes h)=f(g,h)$ for every $g\in G$ and $h\in H$. 
\begin{lemma}\label{tensor-is-commutative}
The image of $\otimes\colon G\times H\rightarrow G\otimes H$ generates $G\otimes H$, the group $G\otimes H$ is abelian, and $\mathpzc{Ab}(G)\otimes \mathpzc{Ab}(H)\cong G\otimes H$ (in particular, $G\otimes H$ is an abelian group). 
\end{lemma}
\begin{proof}
The set $G\times H$ generates the free group $F(G\times H)$, and the natural map $F(G\times H)\rightarrow G\otimes H$ is onto. Then, the group $G\otimes H$ is generated by the image of $G\times H$. Let $f\colon G\times H\rightarrow K$ be a bilinear map, where $K$ is another group (say in additive notation even if it is not assumed to be commutative).  Let $g,g^{\prime}\in G$, $h,h^{\prime}\in H$. We have $f(g,h)+f(g,h^{\prime})+f(g^{\prime},h)+f(g^{\prime},h^{\prime})=f(g,hh^{\prime})+f(g^{\prime},hh^{\prime})=f(gg^{\prime},hh^{\prime})=f(gg^{\prime},h)+f(gg^{\prime},h^{\prime})=f(g,h)+f(g^{\prime},h)+f(g,h^{\prime})+f(g^{\prime},h^{\prime})$ so that $f(g,h^{\prime})+f(g^{\prime},h)=f(g^{\prime},h)+f(g,h^{\prime})$. Thus any two elements of the image of $f$ commute, so the image of $f$ generates a commutative subgroup of $K$. This proves that $G\otimes H$ is abelian. Let $\gamma\colon G\times H\rightarrow \mathpzc{Ab}(G)\times \mathpzc{Ab}(H)$ be the canonical map which is onto. It is clear that if $f\in \mathpzc{Bil}(G\times H,K)$, then $f\circ \gamma\in \mathpzc{Bil}(\mathpzc{Ab}(G)\times \mathpzc{Ab}(H),K)$. We thus define a map $\Psi\colon \mathpzc{Bil}(G\times H,K)\rightarrow \mathpzc{Bil}(\mathpzc{Ab}(G)\times \mathpzc{Ab}(H),K)$ by $\Psi(f)=f\circ \gamma$. It turns that it is one-to-one (since $\gamma$ is onto). Because the image of $f\in \mathpzc{Bil}(G\times H,K)$ generates an abelian subgroup in $K$, for a fixed $h\in H$ the kernel of the homomorphism $f(\cdot,h)\colon g\in G\rightarrow f(g,h)\in K$ contains $[G,G]$ in such a way that $f(g,h)=f(g^{\prime},h)$ for every $g^{\prime}g^{-1}\in [G,G]$. The same holds for $f(g,\cdot)\colon h\in H\mapsto f(g,h)\in K$ for all fixed $g\in G$. This implies that $f=f^{\prime}\circ\gamma$ for some $f^{\prime}\colon \mathpzc{Ab}(G)\times\mathpzc{Ab}(H)\rightarrow K$. Because $\gamma$ is onto, it can be checked that $f^{\prime}$ is bilinear. It follows that $\Psi$ is a bijection, and it is even natural in $K$ (see~\cite{MacLane}). The last statement then follows from usual category theoretic arguments. \qed
\end{proof}
It follows from lemma~\ref{tensor-is-commutative} that it is unnecessary to consider tensor product for non-abelian groups. This is the reason why bilinear maps are defined on abelian groups. 

More generally, it is also possible to define the tensor product of $R$-modules. Let $A,B$ be two $R$-modules (in additive notation), and let $C$ be the submodule of $R^{(A\times B)}$ generated by the elements  $(a+a^{\prime},b)-(a,b)-(a^{\prime},b)$, $(a,b+b^{\prime})-(a,b)-(a,b^{\prime})$, $(\alpha a,b)-\alpha(a,b)$, $(a,\alpha b)-\alpha(a,b)$ for all $a,a^{\prime}\in A$, $b,b^{\prime}\in B$, $\alpha\in R$. Then, the quotient $R$-module $R^{(A\times B)}/C$  is called the \emph{tensor product} of $A$ and $B$, and is denoted by $A\otimes_R B$. It is the solution of the following universal problem: let $D$ be a $R$-module, and let $f\in \mathpzc{Bil}_R(A\times B,D)$. Then, there exists a unique $R$-linear map $g\colon A\otimes_R B\rightarrow D$ such that $g(a\otimes b)=f(a,b)$ for all $a\in A$, $b\in B$ (where $\otimes \colon A\times B\rightarrow A\otimes_R B$ is the restriction to $A\times B$ of the canonical epimorphism from $R^{(A\times B)}$ to $A\otimes_R B$, and it is actually a $R$-bilinear map also called the \emph{canonical bilinear map}). 
\begin{remark}
Taking $R$ to be $\mathbb{Z}$, then we recover the tensor product of abelian groups and it follows that $\mathpzc{Ab}(G)\otimes_{\mathbb{Z}}\mathpzc{Ab}(H)\cong G\otimes H\cong \mathpzc{Ab}(G)\otimes\mathpzc{Ab}(H)$ for every groups $G,H$. Moreover the maps $(a,b)\in A\times B\mapsto a\otimes b\in A\otimes B$ and $(a,b)\in A\times B\mapsto a\otimes b\in A\otimes_{\mathbb{Z}} B$ are also (essentially) the same, where $A,B$ are abelian groups. 
\end{remark}
In what follows, if $A,B$ are two abelian groups, then $A\otimes B$ stands for $A\otimes_{\mathbb{Z}}B$ (according to the above remark there is no confusion).

It is clear by construction that $A\otimes_R B$ is spanned as a $R$-module by $a\otimes b$ where $(a,b)\in A\times B$. Therefore any element of $A\otimes_R B$ is given as a finite sum $\alpha_1(a_1\otimes b_1)+\cdots+\alpha_n(a_n\otimes b_n)$, $\alpha_i\in R$, $a_i\in A$, $b_i\in B$, $i=1,\cdots,n$. Such elements are referred to as \emph{tensors} while generating elements of the form $a\otimes b$ are called \emph{elementary} (or \emph{basic}) \emph{tensors}.

Other properties of the tensor product are recalled below. The first result is given without proof since it is easy. The proofs of the two other may be found for instance in~\cite{BouAlg}. 
\begin{lemma}\label{spanning}
Let $A$ and $B$ be two $R$-modules with respective spanning sets $S$ and $T$. Then, $A\otimes_R B$ is generated as a $R$-module by the basic tensors $s\otimes t$, $s\in S$, $t\in T$.
\end{lemma}
\begin{lemma}\label{symmetric}
Let $A,B$ be two $R$-modules. There is a unique isomorphism of $R$-modules $\sigma\colon A\otimes_R B\cong B\otimes_R A$ such that $\sigma(a\otimes b)=(b\otimes a)$ for every $a\in A$, $b\in B$.
\end{lemma}
\begin{lemma}\label{general-distributivity}
Let $A$ be a $R$-module, and $(B_i)_{i\in I}$ be a finite family of $R$-modules (\emph{i.e.}, $I$ is assumed to be a finite set). Then, there is a unique isomorphism of $R$-modules $\delta\colon A\otimes_R \bigoplus_{i\in I}B_i\cong \bigoplus_{i\in I}(A\otimes_R B_i)$ such that $\delta(a\otimes (b_i)_{i\in I})=(a\otimes b_i)_{i\in I}$ for every $a\in A$ and $(b_i)_{i\in I}\in \bigoplus_{i\in I}B_i$. 
\end{lemma}
%\begin{proof}
%Let us define $f\colon A\times \bigoplus_{i\in I}B_i\rightarrow \bigoplus_{i\in I}(A\otimes_R B_i)$ by $f(a,(b_i)_{i\in I})=(a\otimes b_i)_{i\in I}$. It is clearly $R$-bilinear, and thus there is a unique $R$-linear map $\delta\colon A\otimes_R \bigoplus_{i\in I}B_i\rightarrow \bigoplus_{i\in I}(A\otimes_R B_i)$ such that $\delta(a\otimes (b_i)_{i\in I})=f(a,(b_i)_{i\in I})$ for every $a\in A$, $(b_i)_i\in \bigoplus_{i\in I} B_i$. Now, for each $i\in I$, we consider $f_i\colon A\times B_i\rightarrow A\otimes_R \bigoplus_{i\in I}B_i$ given by $f_i(a,b)=(a\otimes\delta_{b_i})$, where $\delta_{b_i}(i)=b_i$ and $\delta_{b_i}(j)=0$ for every $j\not=i$. Since $f_i$ is $R$-bilinear, it gives rise to a $R$-linear map $\gamma_i\colon A\otimes_R B_i\rightarrow A\otimes_R \bigoplus_{i\in I}B_i$ with the convenient property. By definition of the direct sum, there is a unique $R$-linear map $\gamma\colon \bigoplus_{i\in I}(A\otimes B_i)\rightarrow A\otimes_R \bigoplus_{i\in I}B_i$ such that $\gamma(a\otimes (b_i)_{i\in I})=(a\otimes b_i)_{i\in I}$. It is now easy to check that $\delta$ and $\gamma$ are inverses one from the other.  
%\qed
%\end{proof}
\begin{remark}
It may be shown that $\otimes_R$ is also ``associative'': $(A\otimes_R B)\otimes_R C\cong A\otimes_R(B\otimes_R C)$ for every $R$-modules $A,B,C$ (this isomorphism is natural in $A,B,C$). Together with lemma~\ref{symmetric}, this shows that the category of $R$-modules with the tensor product is a symmetric monoidal category (see~\cite{MacLane}). Loosely speaking this means that the bracketing of factors in a $n$-fold tensor product is irrelevant (because any two $n$-fold tensor products that differ only  in the position of brackets are canonically isomorphic). The notion of multilinear maps $f\colon A_1\times\cdots \times A_n\rightarrow B$ (where the $A_i$'s and $B$ are $R$-modules) (see for instance~\cite{Boneh,Coron,Garg,Huang,Rothblumn}) is equivalent to that of linear maps $f\colon A_1\otimes_R\cdots \otimes_R A_n\rightarrow B$ (bilinear maps are recovered with $n=2$).  In particular, any such multilinear map is actually induced by a unique  bilinear map, for instance $f\colon A_1\times (A_2\otimes_R\cdots \otimes_R A_n)\rightarrow B$. We take advantage of this remark to indicate that the notion of tensor product was already used in~\cite{Boneh} (remark 7.1 and subsection 7.2) but in a somewhat limited way since it was not the purpose of the authors. In this contribution we limit ourselves to bilinear maps.
\end{remark}

\section{Tensor product of finite abelian groups}\label{finiteAbgroups}

In this section we focus on the tensor product of finite abelian groups that is even explicitly computed. Moreover we give some conditions under which a pairing may exist.  
\subsection{Some computations of tensor products}
The objective of this subsection is to compute the tensor product of finite abelian groups. So it seems natural to compute at first the easiest example. In what follows, $(a,b)$ denotes the greatest common divisor of $a$ and $b$. The cyclic group of integers modulo $a$ is denoted by $\mathbb{Z}_a$ (and also $C_a$ when considered multiplicatively written). 
\begin{lemma}\label{easycase}
For every positive integers $a,b$, $\mathbb{Z}_a\otimes \mathbb{Z}_b\cong \mathbb{Z}_{(a,b)}$
\end{lemma}
\begin{proof}
Since $(a,b)$ divides both $a$ and $b$, the map $f\colon \mathbb{Z}_a\times \mathbb{Z}_b\rightarrow \mathbb{Z}_{(a,b)}$ given by $f(x\bmod a,y\bmod b)=(xy)\bmod(a,b)$ is well-defined. Moreover it is bilinear so that it gives rise to a group homomorphism $\pi\colon \mathbb{Z}_a\otimes \mathbb{Z}_b\rightarrow\mathbb{Z}_{(a,b)}$ such that $\pi((x\bmod a)\otimes (y\bmod b))=(xy)\bmod(a,b)$. We observe that $\pi((x\bmod a)\otimes 1)=x\bmod (a,b)$ for every $x$, so that $\pi$ is onto. Let $\mathbb{Z}\rightarrow \mathbb{Z}_a\otimes \mathbb{Z}_b$ be given by $x\mapsto x(1\otimes 1)$. This is clearly a homomorphism of groups, and for $x\in a\mathbb{Z}$, we have $x(1\otimes 1)=((x\bmod a)\otimes 1)=0$. Similarly, when $x\in b\mathbb{Z}$, we have $x(1\otimes 1)=1\otimes (x\bmod b)=0$. Therefore, $a\mathbb{Z}+b\mathbb{Z}=(a,b)\mathbb{Z}$ belongs to its kernel, and we obtain a homomorphism of groups $g\colon \mathbb{Z}_{(a,b)}\rightarrow \mathbb{Z}_a\otimes\mathbb{Z}_b$ such that $g(x\bmod (a,b))=x(1\otimes 1)=((x\bmod a)\otimes 1=1\otimes(x\bmod b)$ for all $x$.  We have $\pi(g(x\bmod (a,b)))=\pi((x\bmod a)\otimes 1)=x\bmod (a,b)$ for every $x$. We have $g(\pi(x(1\otimes 1)))=xg(1)=x(1\otimes 1)$. This is sufficient to check that $\pi$ and $g$ are inverses one from the other because all tensors in $\mathbb{Z}_a\otimes\mathbb{Z}_b$ have the form $x(1\otimes 1)$ for some $x\in \mathbb{Z}$. Indeed, for an elementary tensor $(x\bmod a)\otimes (y\bmod b)=(xy)(1\otimes 1)$. So sums of elementary tensors are also multiple of $1\otimes 1$. \qed
\end{proof}
\begin{remark}
It follows from lemma~\ref{easycase} that $\mathbb{Z}_a\otimes\mathbb{Z}_b=(0)$ if, and only if,  $a$ and $b$ are co-prime. 
\end{remark}
Lemmas~\ref{symmetric}, \ref{general-distributivity} and~\ref{easycase} imply the following result that actually covers all examples of finite abelian groups.
\begin{lemma}\label{generalcase}
Let $(a_i)_{i=1}^m$, and $(b_j)_{j=1}^n$ be two families of positive integers. 
Let $A=\bigoplus_{i=1}^m\mathbb{Z}_{a_i}$, and $B=\bigoplus_{j=1}^n \mathbb{Z}_{b_j}$. Then, $A\otimes B\cong \bigoplus_{\substack{i=1,\cdots,m\\j=1,\cdots,n}}\mathbb{Z}_{(a_i,b_j)}$ where the isomorphism is given by the unique group homomorphism such that $((x_1\bmod a_1,\cdots,x_m\bmod a_m)\otimes (y_1\bmod b_1,\cdots,y_n\bmod b_n))\mapsto ((x_iy_j)\bmod (a_i,b_j))_{\substack{i=1,\cdots,m\\j=1,\cdots,n}}$. Moreover, the canonical bilinear map $\otimes$ is then given by $\displaystyle\otimes\colon A\times B\rightarrow \bigoplus_{\substack{i=1,\cdots,m\\j=1,\cdots,n}}\mathbb{Z}_{(a_i,b_j)}$ with $(x_1\bmod a_1,\cdots,x_m\bmod a_m)\otimes (y_1\bmod b_1,\cdots,y_n\bmod b_n)=((x_iy_j)\bmod (a_i,b_j))_{\substack{i=1,\cdots,m\\j=1,\cdots,n}}$.
\end{lemma}
%\begin{remark}
%Using the same notations as in lemma~\ref{generalcase}, $A\otimes B\cong (0)$ if, and only if, $a_i$ and $b_j$ are co-prime for every $i=1,\cdots,m$ and every $j=1,\cdots,n$. Moreover, $\mathbb{Z}_a^m\otimes \mathbb{Z}_a^m\cong \mathbb{Z}_a^{m^2}$, and if $(a_i)_{i=1}^m$ is a family of positive integers such that $(a_i,a_j)=1$ for every $i\not=j$, then $A\otimes A\cong A$ where $A=\bigoplus_{i=1}^m\mathbb{Z}_{a_i}$. 
%\end{remark}
Any finite abelian group $A$ is isomorphic to a direct product $\bigoplus_{p\in P}\left (\bigoplus_{i\in A_p}\mathbb{Z}_{p^{i}}\right)$ where $P$ is the set of all prime numbers, for each $p\in P$, $A_p$ is a finite subset of $\mathbb{N}_{+}$ such that all but finitely many $A_p$'s are non-void  (hence  $\bigoplus_{p\in P}\left (\bigoplus_{i\in A_p}\mathbb{Z}_{p^{i}}\right)$ is a false infinite sum since $\bigoplus_{i\in \emptyset}\mathbb{Z}_{p^{i}}\cong (0)$ for every $p\in P$ with $A_p=\emptyset$). This decomposition is unique up to isomorphism, and we refer to it as the \emph{primary decomposition}. All these results make possible to compute $A\otimes B$ for any finite abelian groups $A,B$ using lemma~\ref{generalcase}, and also to deduce an essential finiteness result for tensor products.
\begin{lemma}\label{finiteness}
The tensor product of two finite groups is finite. Moreover, using the above notations, for every finite abelian groups $A,B$, the primary decomposition of $A\otimes B$ is given by $$A\otimes B\cong \bigoplus_{p\in P}\left (\bigoplus_{i\in A_p}\bigoplus_{j\in B_p}\mathbb{Z}_{p^{\min(i,j)}}\right)\ .$$ 
\end{lemma}
\begin{proof}
Let $A$ and $B$ be two finite abelian groups. Then each of them admits a decomposition in direct sum of finite cyclic groups, and their tensor product is finite according to lemma~\ref{generalcase}. Because the tensor product of two groups is isomorphic to the tensor product of their abelianization (lemma~\ref{tensor-is-commutative}), the expected conclusion holds (we implicitly used the two easy facts that the abelianization of a finite group is finite, and the isomorphism relation of groups preserves the order). \qed
\end{proof}
\begin{remark}
The tensor product of two finite groups does not depend on the decomposition of the abelianization of each group into a direct sum of cyclic groups. Indeed,  $\otimes$ is a functor (and even a bifunctor), and it is an obvious property of functors to transform isomorphisms into isomorphisms. More precisely, for groups (finite or not) $G,G^{\prime},H,H^{\prime}$ such that $G\cong G^{\prime}$ and $H\cong H^{\prime}$, then $G\otimes H\cong G^{\prime}\otimes H^{\prime}$. The converse assertion is false since for instance $\mathbb{Z}_{6}\otimes\mathbb{Z}_{4}\cong\mathbb{Z}_2\cong \mathbb{Z}_2\otimes\mathbb{Z}_2$.
\end{remark}

\subsection{Non-degeneracy of the canonical bilinear map}

In this subsection we present a sufficient condition for the canonical bilinear map $\otimes$ to be non-degenerate. We also prove that the canonical bilinear map from $A\times A$ to the tensor square $A\otimes A$ always is non-degenerate for each finite abelian group $A$, providing an infinite family of pairings. 
\begin{lemma}\label{nondegeneracy1}
Let $a$ and $b$ be two positive integers. The canonical bilinear map $\otimes\colon (x\bmod a,y\bmod b)\in \mathbb{Z}_a\times\mathbb{Z}_b\rightarrow (xy)\bmod (a,b)\in \mathbb{Z}_{(a,b)}\cong \mathbb{Z}_a\otimes\mathbb{Z}_b$ is non-degenerate if, and only if, $a=b$. 
\end{lemma}
\begin{proof}
If $a=b=1$, then all groups are trivial, and the result is obvious. Let $a=(a,b)=b\not=1$. Let $x\bmod a\not=0$ such that for every $y$, $(xy)\bmod a=0$, then we obtain a contradiction when $y=1$. Therefore, $\otimes$ is non-degenerate. Now, let us assume that $(a,b)<a$ for instance. Then, $(a,b)\bmod a\not=0$, and for all $y$, $(a,b)y\bmod (a,b)=0$ so that $\otimes$ is degenerate. \qed
\end{proof}
Let $p$ be a prime number. Let $A,B$ be two finite abelian $p$-groups (that is, finite abelian groups of order $p^n$ for some $n$), and $C$ be an abelian group. The theorem of invariant factors imply that  $A\cong \displaystyle\bigoplus_{i=1}^m\mathbb{Z}_{p^{\alpha_i}}$ with $\alpha_1\geq \alpha_2\geq \cdots \geq \alpha_m\geq 1$, and $B\cong \displaystyle\bigoplus_{j=1}^n\mathbb{Z}_{p^{\beta_j}}$ with $\beta_1\geq \beta_2\geq \cdots \geq \beta_n\geq 1$. We recall that the \emph{exponent} $\mathpzc{exp}(G)$ of a finite group $G$ is the least common multiple of the orders of the elements of $G$. Therefore, $\mathpzc{exp}(A)=p^{\alpha_1}$ and $\mathpzc{exp}(B)=p^{\beta_1}$. Let $f\colon A\times B\rightarrow C$ be a pairing. Let us assume for instance that $\mathpzc{exp}(A)>\mathpzc{exp}(B)$. Then, $f((\mathpzc{exp}(B)1,0,\cdots,0),(y_1,\cdots,y_n))=f((1,0,\cdots,0),\mathpzc{exp}(B)(y_1,\cdots,y_n))=f((1,0,\cdots,0),(0,\cdots,0))=1$ for every $y_1,\cdots,y_n$. Thus $f$ would be degenerate. Therefore, $\mathpzc{exp}(A)=\mathpzc{exp}(B)$. The following lemma is thus proved.
\begin{lemma}\label{exponent}
Let $A,B$ be two finite abelian $p$-groups, and $C$ be an abelian group. Let $f\colon A\times B\rightarrow C$ be a pairing. Then, $\mathpzc{exp}(A)=\mathpzc{exp}(B)$.
\end{lemma}
Let $A\cong \bigoplus_{p\in P}A(p)$ and $B\cong\bigoplus_{p\in P}B(p)$ be the primary decomposition of two finite abelian groups  $A$ and $B$. For each prime number $p$, let $A(p)\cong \bigoplus_{i=1}^{n_A(p)}\mathbb{Z}_{p^{\alpha_i}}$ and $B(p)\cong \bigoplus_{j=1}^{n_B(p)}\mathbb{Z}_{p^{\beta_j}}$ be the invariant factor decomposition of each factor of the primary decomposition.  Let $p,q$ be two distinct prime numbers. Then by lemma~\ref{generalcase}, $A(p)\otimes B(q)\cong \bigoplus_{i=1}^{n_A(p)}\bigoplus_{j=1}^{n_B(p)}\mathbb{Z}_{(p^{\alpha_i},q^{\beta_j})}\cong (0)$. Thus, again by lemma~\ref{generalcase}, $A\otimes B\cong \displaystyle\bigoplus_{p\in P}A(p)\otimes B(p)$. Let $f\colon A \times B\rightarrow A\otimes B$ be a bilinear map. Then, $f(a,b)=0$ for every $a\in A(p)$, $b\in B(q)$ with distinct prime numbers $p,q$. For each prime number $p$, let $f_p\colon A(p)\times B(p)\rightarrow A(p)\otimes B(p)$ be the obvious restriction of $f$, which also is a bilinear map. Then, $f((a_p)_{p\in P},(b_p)_{p\in P})=(f_p(a_p,b_p))_{p\in P}$.  In particular $\otimes_p\colon A(p)\times B(p)\rightarrow A(p)\otimes B(p)$ is the corresponding canonical bilinear map so that $(a_p)_{p\in P}\otimes (b_p)_{p\in P}=(a_p\otimes_p b_p)_{p\in P}$. 
\begin{theorem}\label{prim-decomp-pairing}
Using the above notations, the canonical bilinear map $\otimes\colon A\times B\rightarrow A\otimes B$ is non-degenerate if, and only if, for every prime number $p$, $\otimes_p$ is non-degenerate (and in particular, according to lemma~\ref{exponent}, $\mathpzc{exp}(A(p))=\mathpzc{exp}(B(p))$). More generally, $f\in\mathpzc{Bil}(A\times B,A\otimes B)$ is a pairing if, and only if, $f_p\in\mathpzc{Bil}(A(p)\times B(p),A(p)\otimes B(p))$ is a pairing for each prime number $p$.
\end{theorem}
\begin{proof}
It is obviously sufficient to prove the second assertion. 
It is clear that non-degeneracy of all $f_p$ implies non-degeneracy of $f$. Now, let us assume that $f$ is non-degenerate but there is some prime number $p_0$ and $a\in A(p_0)$, $a\not=0$ such that $f_p(a, b)=0$ for every $b\in B(p)$. Then, let us consider $(a_p)_{p\in P}\in A$ such that $a_p=0$ for every $p\not=p_0$, and $a_{p_0}=a$. Then, for every $(b_p)_{p\in P}\in B$, $f((a_p)_{p},(b_p)_p)=0$ which contradicts non-degeneracy of $f$. \qed
\end{proof}
Next lemma explains in what extend non-degeneracy of the canonical bilinear map is essential for the existence of pairings. 
\begin{lemma}\label{degeneracy-of-can-implies-no-pairings}
Let $A,B$ be two non-trivial $R$-modules. If  there are a $R$-module $C$ and a pairing $\langle \cdot\mid\cdot\rangle\colon A\times B\rightarrow C$, then the canonical bilinear map $\otimes \colon A\times B\rightarrow A\otimes_R B$ is non-degenerate.  
\end{lemma}
\begin{proof}
By contraposition, let us assume that $\otimes\colon A\times B\rightarrow A\otimes_R B$ is degenerate, and for instance that it is not left non-degenerate. Then, there exists $a_0\in A$, $a_0\not=0_A$, such that for every $b\in B$, $a_0\otimes b=0$. Let $\langle\cdot\mid\cdot\rangle\colon A\times B\rightarrow C$ be a bilinear map. Then, there exists a unique $R$-linear map $f\colon A\otimes_R B\rightarrow C$ such that $f(a\otimes b)=\langle a\mid b\rangle$. In particular, $\langle a_0\mid b\rangle=f(a_0\otimes b)=f(0)=1_C$ for all $b\in B$. Therefore, $\langle\cdot\mid\cdot\rangle$ is left degenerate. \qed
\end{proof}
%Lemma~\ref{degeneracy-of-can-implies-no-pairings} states that for two non-trivial $R$-modules which have a degenerate canonical bilinear map, there is no hope to construct a pairing from their cartesian product. For instance, let $a,b,d$ be positive integers $>1$, $(a,b)=1$, then $\otimes \colon \mathbb{Z}_{da}\times \mathbb{Z}_{db}\rightarrow \mathbb{Z}_{d}\cong \mathbb{Z}_{da}\otimes\mathbb{Z}_{db}$ given by $(x\bmod da)\otimes (y\bmod db)=(xy)\bmod d$ is degenerate (for instance, $(d\bmod da)\otimes (y\bmod db)=(dy)\bmod d=0$ for every $y$). Therefore there are no pairings $\langle\cdot\mid\cdot\rangle\colon \mathbb{Z}_{da}\times\mathbb{Z}_{db}\rightarrow C$ for any abelian group $C$.  
We anticipate a result from subsection~\ref{duality} to state a sufficient condition for the existence of a pairing, from the cartesian square to the tensor square of some abelian group, provided by the following result.
\begin{theorem}\label{tensor-square}
Let $A$ be a finite abelian group. Then, the canonical bilinear map $\otimes\colon A\times A\rightarrow A\otimes A$ is non-degenerate.
\end{theorem}
\begin{proof}
Since $A$ is a finite abelian group, it admits a decomposition into cyclic groups $A\cong\displaystyle\bigoplus_{i=1}^n \mathbb{Z}_{d_i}$ for some integers $d_i$. In subsection~\ref{duality} is proved that there exists at least one pairing $(\bigoplus_{i=1}^n \mathbb{Z}_{d_i})\times (\bigoplus_{i=1}^n \mathbb{Z}_{d_i})\rightarrow C_N$, where $C_N$ denotes the cyclic group of order $N$, with $N=\prod_{i=1}^n d_i$. Therefore according to lemma~\ref{degeneracy-of-can-implies-no-pairings}, the canonical bilinear map $\otimes\colon (\bigoplus_{i=1}^n \mathbb{Z}_{d_i})\times (\bigoplus_{i=1}^n \mathbb{Z}_{d_i})\rightarrow (\bigoplus_{i=1}^n \mathbb{Z}_{d_i})\otimes(\bigoplus_{i=1}^n \mathbb{Z}_{d_i})$ is non-degenerate. Let $\phi\colon A \rightarrow \bigoplus_{i=1}^n \mathbb{Z}_{d_i}$ be an isomorphism of groups. Let us assume that there exists $a_0\in A$ such that $a_0\otimes a=0$ for every $a\in A$. Then, $\phi(a_0)\otimes\phi(a)=0$ for every $a\in A$. Since $\phi$ is onto, this implies that $\phi(a_0)\otimes a^{\prime}=0$ for every $a^{\prime}\in \bigoplus_{i=1}^n \mathbb{Z}_{d_i}$. So that $\phi(a_0)=0$ (by non-degeneracy), and thus $a_0=0_A$. \qed
\end{proof}
\begin{remark}
Theorem~\ref{tensor-square} provides an infinite family of pairings because in this situation $\otimes$ is itself a pairing. This generalizes some optimized constructions of pairings over elliptic curves on finite fields as defined in~\cite{Hess} such as Weil (\cite{Miller}), Tate (\cite{Ruck}) and ate (\cite{Hessetal}) pairings which may be defined on $\mathbb{Z}_a\times \mathbb{Z}_a$ for some integer $a$ and with values in the group of $a$-th roots of the unity $\mu_a\cong \mathbb{Z}_a\cong \mathbb{Z}_a\otimes\mathbb{Z}_a$ in a finite field $\mathbb{F}_{p^n}$ (where $a$ divides $p^n-1$). We observe however that these pairings are usually defined on a bigger cartesian product of groups (see for instance~\cite{Silverman} concerning Weil pairing).
\end{remark}
\begin{remark}
We also observe that there are some pairings $f\colon A\times B\rightarrow C$ where $A,B,C$ are finite abelian groups such that $A$ and $B$ are non-isomorphic. For instance, let $p$ be a prime number, and $m>1$ be an integer. Then, the canonical bilinear map $\otimes\colon \mathbb{Z}_p^m \times \mathbb{Z}_p\rightarrow \mathbb{Z}_p^{m}$ given by $(x_i\bmod p)_{i=1}^m\otimes (y\bmod p)=(x_i y\bmod p)_{i=1}^m$ is non-degenerate. 
\end{remark}

\section{Constructions of bilinear maps and pairings}\label{constructions}

In full generality the canonical bilinear map is not always non-degenerate  (even when the tensor product does not collapse to zero, see the discussion after lemma~\ref{degeneracy-of-can-implies-no-pairings}). As stated in lemma~\ref{degeneracy-of-can-implies-no-pairings}, non-degeneracy of this function is a necessary condition for the existence of pairings. In this section, we present other constructions of bilinear maps and pairings using the fact that the set of all bilinear maps, from some fixed $A\times B$ to $C$, forms an abelian group (or $R$-module). Moreover we prove that for a particular choice of $A,B$ and $C$, $\mathpzc{Bil}(A\times B,C)$ is actually a ring, and that the pairings are exactly the group of units of this ring (see theorem~\ref{grpofunits}). 

\subsection{Abelian group structure of bilinear maps (and pairings)}

First of all, we know from the proof of lemma~\ref{tensor-is-commutative} that for every groups $G,H,K$, $\mathpzc{Bil}(G\times H,K)\cong \mathpzc{Bil}(\mathpzc{Ab}(G)\times\mathpzc{Ab}(H),K)$. According to the universal property of tensor product of groups, $\mathpzc{Bil}(G\times H,K)\cong \mathpzc{Hom}(G\otimes H,K)$. Therefore, for every triple of abelian groups (respectively, $R$-modules) $A,B,C$, $\mathpzc{Bil}(A\times B,C)\cong \mathpzc{Hom}(A\otimes B,C)$ (respectively, $\mathpzc{Bil}_R(A\times B,C)\cong \mathpzc{Hom}_R(A\otimes_R B,C)$). But the later is itself an abelian group (respectively, a $R$-module) with point-wise operations, so that $\mathpzc{Bil}(A\times B,C)$ (respectively, $\mathpzc{Bil}_R(A\times B,C)$) becomes an abelian group (respectively, a $R$-module). More precisely, let $A,B,C$ be three $R$-modules, and let us assume that $A,B$ are given in additive notation (recall that $0_A,0_B$ are the identity elements of $A$ and $B$) and $C$ is multiplicatively written (recall that $1_C$ is the identity of $C$), we have for $f,g\in \mathpzc{Bil}_R(A\times B,C)$ and $\alpha\in R$, three new bilinear maps $fg$, $f^{-1}$, $f^{\alpha}\in  \mathpzc{Bil}_R(A\times B,C)$ defined by $(fg)(a,b)=f(a,b)g(a,b)$, $f^{-1}(a,b)=(f(a,b))^{-1}$ and $f^{\alpha}(a)=(f(a))^{\alpha}$ (where the scalar multiplication in $C$ is given by $(\alpha,c)\mapsto c^{\alpha}$ because $C$ is assumed to be in multiplicative notation) for all $a\in A$, $b\in B$. This also defines a structure of $\mathbb{Z}$-module given by $f^n(a,b)=(f(a,b))^n$ for all $a\in A$, $b\in B$, $n\in\mathbb{Z}$. 
%\begin{lemma}\label{product}
%Let $f,g$ be two bilinear maps. If for every $a\not=0_A$, there exists $b\in B$ such that $g(a,b)\not=f(a,b)^{-1}$, then $fg$ is left non-degenerate. The symmetric result holds for right non-degeneracy.
%\end{lemma}
%\begin{proof}
%Let $a\in A$, and let us assume that $f(a,b)g(a,b)=1_C$ for every $b\in B$. Then, $g(a,b)=f(a,b)^{-1}$ for all $b$ so that $a=0_A$. \qed
%\end{proof}
%Next result is obvious.
%\begin{lemma}\label{inverse}
%Let us assume that $f$ is a left (respectively, right) non-degenerate bilinear map from $A\times B$ to $C$, then $f^{-1}$ also is left (respectively, right) non-degenerate. More generally, if $g\colon C\rightarrow D$ is a monomorphism of $R$-modules and $f$ is left (respectively, right) non-degenerate, then $g\circ f\colon A\times B\rightarrow D$ is a bilinear map which is also left (respectively, right) non-degenerate.
%\end{lemma}
The following (obvious) construction uses direct product of modules (or abelian groups).
\begin{lemma}\label{direct-prod}
Let $(C_i)_{i=1}^n$ be a family of $R$-modules, and $A,B$ be $R$-modules. Let $f_i\in\mathpzc{Bil}_{R}(A\times B,C_i)$ for $i=1,\cdots,n$. Then, the map $(f_1,\cdots,f_n)\colon A\times B\rightarrow C_1\times\cdots\times C_n$ defined by $(f_1,\cdots,f_n)(a,b)=(f_1(a,b),\cdots,f_n(a,b))$  belongs to $\mathpzc{Bil}_R(A\times B,C_1\times\cdots\times C_n)$. Moreover, if at least one of the $f_i$'s is non-degenerate, then $(f_1,\cdots,f_n)$ itself is non-degenerate.
\end{lemma}
Let us study the group structure of $\mathpzc{Bil}(A\times B,C)$ in an easy case. 
For every group  $G$, let $\mathpzc{End}(G)=\mathpzc{Hom}(G,G)$ which is a ring when $G$ is abelian. Let $a,b$ be two positive integers. Then, we have the following sequence of group isomorphisms $\mathpzc{Bil}(\mathbb{Z}_a\times\mathbb{Z}_b,\mathbb{Z}_{(a,b)})\cong \mathpzc{Hom}(\mathbb{Z}_{a}\otimes\mathbb{Z}_b,\mathbb{Z}_{(a,b)})\cong \mathpzc{End}(\mathbb{Z}_{(a,b)})$. It easy to check that $\mathpzc{End}(\mathbb{Z}_n)\cong\mathbb{Z}_n$ as rings for any $n$. So that $\mathpzc{Bil}(\mathbb{Z}_a\times\mathbb{Z}_b,\mathbb{Z}_{(a,b)})$ may also be equipped with a structure of commutative ring with a unit isomorphic to $\mathbb{Z}_{(a,b)}$. Moreover, as a cyclic group of order $(a,b)$, and therefore as a $\mathbb{Z}$-module, $\mathpzc{Bil}(\mathbb{Z}_a\times\mathbb{Z}_b,\mathbb{Z}_{(a,b)})$ is generated by the canonical bilinear map $\otimes$. Thus $\mathpzc{Bil}(\mathbb{Z}_a\times\mathbb{Z}_b,\mathbb{Z}_{(a,b)})$ is the free $\mathbb{Z}_{(a,b)}$-module generated by $\otimes$, or, in other terms, it is isomorphic to the group $\mathbb{Z}_{(a,b)}$, so that for any bilinear map $f\colon \mathbb{Z}_{a}\times\mathbb{Z}_b\rightarrow \mathbb{Z}_{(a,b)}$, there exists a unique $k_f\in\mathbb{Z}_{(a,b)}$ such that $f=\otimes^{k_f}$, where we recall that $\otimes^{k_f}(x\bmod a,y\bmod b)=k_f xy\bmod (a,b)$. It follows that if $\otimes$ is non-degenerate, then $f$ is a pairing if, and only if, $(k_f,(a,b))=1$. Moreover, if $\otimes$ is degenerate, then there is no pairing defined on $\mathbb{Z}_a\times\mathbb{Z}_b$ by lemma~\ref{degeneracy-of-can-implies-no-pairings}. Thus, according to lemma~\ref{nondegeneracy1}, a pairing $f\in \mathpzc{Bil}(\mathbb{Z}_a\times\mathbb{Z}_a,\mathbb{Z}_{a})$ is exactly a   generator of the cyclic group $\mathpzc{Bil}(\mathbb{Z}_a\times\mathbb{Z}_a,\mathbb{Z}_{a})$ of order $a$.  The following result is proved.
\begin{theorem}\label{grpofunits}
The set of pairings from $\mathbb{Z}_{a}\times \mathbb{Z}_a$ to $\mathbb{Z}_{a}$ forms a group isomorphic to the group of invertible elements of the ring $\mathbb{Z}_{a}$ under multiplication. In particular, there are exactly $\phi(a)$ pairings in this situation, and if $f\in \mathpzc{Bil}(\mathbb{Z}_a\times\mathbb{Z}_a,\mathbb{Z}_{a})$ is a pairing, then any other pairing $g$ has the form $f^{k_g}$, for a unique $k_g\in\mathbb{Z}_{a}$ invertible modulo $a$. Moreover, if $p$ is a prime number, then the group of pairings in $\mathpzc{Bil}(\mathbb{Z}_p\times\mathbb{Z}_p,\mathbb{Z}_{p})$ is isomorphic to $\mathbb{Z}_{p}^*$. 
\end{theorem}
\begin{remark}
Let $p$ be a prime number. Let $f\in \mathpzc{Bil}(\mathbb{Z}_p\times\mathbb{Z}_p,\mathbb{Z}_p)$ be a pairing. According to theorem~\ref{grpofunits}, any other pairing is given by $f^k$ for $k\in\mathbb{Z}_p^*$ as it was already noticed in~\cite{Boxall} (but we observe that the underlying group structure on pairings was not explicitly mentioned). In this situation, the integer $k$ was called the \emph{logarithm} of the pairing \emph{to the base $f$}. This also explains why F. Vercauteren write in~\cite{Vercauteren} that ``there is essentially only one pairing''. 
\end{remark}
%\begin{proposition}
%Let $a$ and $n$ be any positive integers. Any pairing from $\mathbb{Z}_a^n\times\mathbb{Z}_a^n$ to $\mathbb{Z}_a$ is a perfect pairing.
%\end{proposition}
%\begin{proof}
%By the universal property of the direct product of groups, $\mathpzc{Hom}(\mathbb{Z}_a^n,\mathbb{Z}_a)\cong \mathpzc{End}(\mathbb{Z}_a)^n\cong \mathbb{Z}_a^n$ (isomorphic groups). Let $f\colon \mathbb{Z}_a^n\times\mathbb{Z}_a^n\rightarrow \mathbb{Z}_a$ be a pairing. By non-degeneracy, the map $x\in \mathbb{Z}_a^n\mapsto f(x,\cdot)\in \mathpzc{Hom}(\mathbb{Z}_a^n,\mathbb{Z}_a)\cong \mathbb{Z}_a^n$ is one-to-one, and therefore it is a bijection. \qed
%\end{proof}
%\begin{remark}
%Weil pairing in its usual version~\cite{Silverman} or in its optimized version~\cite{Hessetal} is thus a perfect pairing. For certain choices of parameters (for instance when we consider the points of some given prime order), the Tate pairing in its usual version~\cite{BSS2005} is also a perfect pairing. Likewise optimized Tate pairing~\cite{Hessetal} is a perfect pairing.  
%\end{remark}
Let $A\cong \bigoplus_{p\in P}A(p)$ and $B\cong \bigoplus_{p\in P}B(p)$ be two finite abelian groups decomposed following the primary decomposition (each $A(p)$ and $B(p)$ are finite abelian $p$-groups).  For each prime number $p$, $\mathpzc{Bil}(A(p)\times B(p),A(p)\otimes B(p))\cong \mathpzc{End}(A(p)\otimes B(p))$ so that $\mathpzc{Bil}(A(p)\times B(p),A(p)\otimes B(p))$ admits a ring structure. From the discussion preceding theorem~\ref{prim-decomp-pairing}, we know that $\mathpzc{Bil}(A\times B,A\otimes B)\cong \bigoplus_{p\in P}\mathpzc{Bil}(A(p)\times B(p),A(p)\otimes B(p))$ (group direct sum). Let us assume that for each prime number $p$, $A(p)=\mathbb{Z}_{p^{n_p}}=B(p)$ (the case $n_p=0$ is necessarily possible in such a way $A(p)=(0)=B(p)$). Then, according to theorem~\ref{grpofunits}, for each $p$, the pairings in $\mathpzc{Bil}(A(p)\times B(p),A(p)\otimes B(p))\cong \mathbb{Z}_{p^{n_p}}$ form the group $\mathbb{Z}_{p^{n_p}}^{\times}$ of invertible elements modulo $p^{n_p}$, and by theorem~\ref{prim-decomp-pairing} the pairings in $\mathpzc{Bil}(A\times B,A\otimes B)\cong \bigoplus_{p\in P}\mathbb{Z}_{p^{n_p}}$ is the group direct sum $\bigoplus_{p\in P}\mathbb{Z}_{p^{n_p}}^{\times}$. Again by theorem~\ref{grpofunits}, if $n_p\in \{\, 0,1\,\}$ for every prime number $p$, then the pairings in $\mathpzc{Bil}(A\times B,A\otimes B)\cong \bigoplus_{p\in P}\mathbb{Z}_{p^{n_p}}$ is the group direct sum $\bigoplus_{p\in P_0}\mathbb{Z}_{p}^*$ (where $P_0=\{\, p\in P\colon n_p=1\,\}$). 

\subsection{Tensor product of linear maps}

One of the main feature of the tensor product that has not been used yet in this contribution  is the fact $\otimes_R$ is a bifunctor. In particular, it transforms a pair of linear maps into one linear map as follows. Let $A_1,A_2,B_1,B_2$ be four $R$-modules. Let $f\colon A_1\rightarrow A_2$, $g\colon B_1\rightarrow B_2$ be two $R$-linear maps. Then, the map $f\otimes g\colon A_1\otimes_R B_1\rightarrow A_2\otimes_R B_2$ defined by $(f\otimes g)(a\otimes b)=f(a)\otimes g(b)$ is a $R$-module map (be careful that the same symbol $\otimes$ denotes the canonical bilinear map $A_i\times B_i\rightarrow A_i\otimes_R B_i$ for $i=1,2$). Since $\mathpzc{Bil}_{R}(A_1\times B_1,A_2\otimes_R B_2)\cong \mathpzc{Hom}_R(A_1\otimes_R B_1,A_2\otimes_R B_2)$, it follows that $f\otimes g$ is induced by a (unique) $R$-bilinear map $h\colon A_1\times B_1\rightarrow A_2\otimes_R B_2$ such that $h(a,b)=f(a)\otimes g(b)$ for $a\in A_1$, $b\in B_1$. 

\begin{lemma}
Let us assume that $f$ is a  monomorphism, $g$ is an epimorphism, and that the canonical $R$-bilinear map $\otimes\colon A_2\times B_2\rightarrow A_2\otimes_R B_2$ is non-degenerate, then $h$ is left non-degenerate.
\end{lemma}
\begin{proof}
Let $a\in A_1$ such that for every $b\in B_1$, $h(a,b)=0$. Then, $f(a)\otimes g(b)=0$ for every $b\in B$. Since $g$ is onto, $f(a)\otimes b^{\prime}=0$ for all $b^{\prime}\in B_2$. Since the canonical bilinear map is non-degenerate, $f(a)=0_{A_2}$, so that $a=0_{A_1}$ because $f$ is one-to-one. \qed
\end{proof}
\subsection{Divide out the kernels}

In this subsection is presented a natural way to construct a pairing from a bilinear map by dividing out two kernels. 

Let $A,B,C$ be three $R$-modules (where $R$ is a commutative ring with a unit). The groups $A,B$ are written additively, while $C$ is given in multiplicative notation. Let $f\colon A\times B\rightarrow C$ be a $R$-bilinear map. We define two linear maps $\gamma_f\colon A\rightarrow \mathpzc{Hom}_R(B,C)$ and $\rho_f\colon B\rightarrow \mathpzc{Hom}_R(A,C)$ given respectively by $\gamma_f(a)=f(a,\cdot)$ and $\rho_f(b)=f(\cdot,b)$. Let us define $L_f=\bigcap_{b\in B}\ker f(\cdot,b)=\ker \gamma_f$, $R_f=\bigcap_{a\in A}\ker f(a,\cdot)=\ker \rho_f$ which are respectively a sub-module of $A$ and a sub-module of $B$ (they are sometimes called the \emph{annihilator} of $A$ and $B$ respectively, see~\cite{MacLane2}). We observe that if $a-a^{\prime}\in L_f$, then for all $b\in B$, $f(a,b)f(a^{\prime},b)^{-1}=f(a-a^{\prime},b)=1_C$ so that $f(a,b)=f(a^{\prime},b)$. Therefore, there is a well-defined $R$-bilinear map $f_1\colon A/L_f\times B\rightarrow C$ such that $f_1(a\bmod L_f,b)=f(a,b)$ for all $a\in A$, $b\in B$. Similarly, we have a well-defined $R$-linear map $f_2\colon A\times B/R_f\rightarrow C$ such that $f_2(a,b\bmod R_f)=f(a,b)$ for all $a\in A$, $b\in B$. The first map is left non-degenerate while the second is right non-degenerate. We may continue the process in order to get a full non-degeneracy. Let $R_{f_1}=\bigcap_{a\bmod L_f\in A/L_f}\ker f_1(a\bmod L_f,\cdot)=R_f$. Similarly we have $L_{f_2}=L_f$. We obtain two well-defined non-degenerate $R$-bilinear maps $f_3,f_4\colon A/L_f\times B/R_f\rightarrow C$ such that 
\begin{equation}
\begin{array}{lll}
f_3(a\bmod L_f,b\bmod R_f)&=&f_1(a\bmod L_f,b)\\
&=&f(a,b)\\
&=&f_2(a,b\bmod R_f)\\
&=&f_4(a\bmod L_f,b\bmod R_f)
\end{array}
\end{equation} 
for each $a\in A$ and $b\in B$. Thus the two pairings are the same one.

When the bilinear map $f$ into consideration is the canonical bilinear map $\otimes \colon A\times B\rightarrow A\otimes_R B$ itself, then we define ${}^{\perp}B=L_{\otimes}$, and $A^{\perp}=R_{\otimes}$. Moreover, let $\lambda\colon A\rightarrow A/{}^{\perp}B$ and $\delta\colon B\rightarrow B/A^{\perp}$ be the canonical epimorphisms. We have a well-defined non-degenerate pairing $\otimes^{\prime}\colon A/{}^{\perp}B\times B/A^{\perp}\rightarrow A\otimes_R B$ such that $\lambda(a)\otimes^{\prime}\delta(b)=a\otimes b$ for every $a\in A$, $b\in B$. Let us define $\tilde{\otimes}=(\lambda\otimes \delta)\circ \otimes^{\prime}\in \mathpzc{Bil}_R(A/{}^{\perp}B\times B/A^{\perp},A/{}^{\perp}B\otimes_R B/A^{\perp})$. It satisfies $(\lambda\otimes \delta)(\lambda(a)\otimes^{\prime}\delta(b))=(\lambda\otimes \delta)(a\otimes b)=\lambda(a)\otimes_2\delta(b)$ for every $a\in A$, $b\in B$, where $\otimes_2\colon A/{}^{\perp}B\times B/A^{\perp}\rightarrow A/{}^{\perp}B\otimes_R B/A^{\perp}$ is the canonical bilinear map. Actually it is quite clear that  $\tilde{\otimes}=\otimes_2$. 
\begin{lemma}
The canonical bilinear map $\otimes_2\colon A/{}^{\perp}B\times B/A^{\perp}\rightarrow A/{}^{\perp}B\otimes_R B/A^{\perp}$ is non-degenerate.
\end{lemma}
\begin{proof}
The bilinear map $\otimes^{\prime}\colon A/{}^{\perp}B\times B/A^{\perp}\rightarrow A\otimes_R B$ is non-degenerate. Then according to lemma~\ref{degeneracy-of-can-implies-no-pairings}, the canonical bilinear map $\otimes_2\colon \colon A/{}^{\perp}B\times B/A^{\perp}\rightarrow A/{}^{\perp}B\otimes_R B/A^{\perp}$ is itself non-degenerate. \qed
\end{proof}
%\begin{remark}
%Let us define the \emph{non-degenerate tensor product} by $A\tilde{\otimes}_R B=A/{}^{\perp}B\otimes_R B/A^{\perp}$. Together with $(\lambda\otimes \delta)\circ \otimes=\otimes_2\circ(\lambda\times \delta)\in\mathpzc{Bil}_{R}(A\times B,A\tilde{\otimes}_R B)$ (where $\lambda\times\delta\colon A\times B\rightarrow A/{}^{\perp}B\times B/A^{\perp}$ is defined by $(\lambda\times \delta)(a,b)=(\lambda(a),\delta(b))$ for every $a\in A$, $b\in B$), it satisfies the following universal problem. Let $f\in\mathpzc{Bil}_R(A\times B,C)$ such that ${}^{\perp}B\subseteq \ker \gamma_f$ and $A^{\perp}\subseteq \ker \rho_f$.  Then, there is a unique $R$-linear map $\tilde{f}\colon A\tilde{\otimes}B\rightarrow C$ such that for every $a\in A$, $b\in B$, $\tilde{f}(\lambda(a)\otimes_2\delta(b))=f(a,b)$. It is also clear that $f^{\prime}(\lambda(a),\delta(b))=f(a,b)$ is a well-defined bilinear map from $A/{}^{\perp}B\times B/A^{\perp}$ to $C$, and the unique linear map $g\colon A\tilde{\otimes}B\rightarrow C$ such that $g(\lambda(a)\otimes_2\delta(b))=f^{\prime}(\lambda(a),\delta(b))$ for every $a\in A$, $b\in B$, is precisely $\tilde{f}$ itself. We finally observe that if ${}^{\perp}B=\ker\gamma_f$ (respectively,  $A^{\perp}=\ker\rho_f$), then $f^{\prime}$ is left non-degenerate (respectively, right non-degenerate).
%\end{remark}

\subsection{Finite abelian group duality and characters}\label{duality}

Let $A,B,C$ be three $R$-modules. One of the main property of a given pairing $\langle\cdot\mid\cdot\rangle\colon A\times B\rightarrow C$ is the non-degeneracy. It exactly states that  $A$ embeds into $\mathpzc{Hom}_R(B,C)$ as a sub-module by $\langle a\mid \cdot\rangle\colon B\rightarrow C$ for each $a\in A$, and that $B$ embeds into $\mathpzc{Hom}_R(A,C)$ also as a sub-module by $\langle\cdot\mid b\rangle\colon A\rightarrow C$ for each $b\in B$. Using this idea we may construct a  pairing. Let $A,C$ be two $R$-modules, and let $B$ be a sub-module of $\mathpzc{Hom}_R(A,C)$. Let $\langle\cdot\mid\cdot\rangle\colon A\times B\rightarrow C$ be defined by $\langle a\mid b\rangle=b(a)$ for every $a\in A$, $b\in B$. By its very definition, this is a $R$-bilinear map which clearly is right non-degenerate. We observe that the elements of $A$ may be seen as linear maps on $B$ as follows: let $a\in A$, and define $\widehat{a}\colon B\rightarrow C$ by $\widehat{a}(b)=b(a)$. The facts that $\widehat{a}\in \mathpzc{Hom}_R(B,C)$  and  $\widehat{(\cdot)}\colon a\in A \rightarrow \widehat{a}\in \mathpzc{Hom}_R(B,C)$ is a homomorphism of groups are easily checked. We say that $A$ \emph{seperates the points of $B$} (following a usual terminology from functional analysis) if $b(a)=0_C$ for every $b\in B$ implies that $a=0_A$. Equivalently, this means that the map $\widehat{a}=\langle a\mid \cdot\rangle$ is one-to-one for every non-zero $a$, so that $A$ embeds into $\mathpzc{Hom}_R(B,C)$ as a sub-module. In this case, and only in this case, $\langle\cdot\mid\cdot\rangle$ as defined above is a pairing. We now propose two actual examples of such a construction. \\

\noindent\textbf{Dot-product construction: } Let $\mathbb{K}$ be any field. Let $V$ be a $d$-dimensional vector space over $\mathbb{K}$. Its \emph{(algebraic) dual} $V^*$ is the vector space $\mathpzc{Hom}_{\mathbb{K}}(V,\mathbb{K})$ of all linear forms. We observe that $V$ separates the points of $V^*$ since if $v\in V$ is non-zero, then it belongs to some basis of $V$ over $\mathbb{K}$ so that we may choose a linear map $\ell\colon V\rightarrow \mathbb{K}$ such that $\ell(v)\not=0$ and $\ell$ takes any value for the other elements of the basis. Therefore the $\mathbb{K}$-bilinear form $\langle\cdot\mid \cdot\rangle\colon V\times V^*\rightarrow \mathbb{K}$ given by $\langle v\mid \ell\rangle=\ell(v)$ is a pairing. Moreover, if $(e_i)_{i=1}^d$ is a basis of $V$ over the base field, then for each $j=1,\cdots,d$, we may define a linear form $e^j\in V^*$ by the relations $e^j(e_i)=1$ if $j=i$, and $0$ otherwise. It turns that $(e^i)_{i=1}^d$ is a basis of $V^*$ over $\mathbb{K}$ called the \emph{dual basis} of $(e_i)_{i=1}^d$, and that $V\cong V^*$ (as vector spaces). Under the isomorphism $e^i\mapsto e_i$, the pairing becomes $\langle v\mid w\rangle=\sum_{i=1}^d v_i w_i$, where $v_i=e^i(v)$, $w_i=e^i(w)$ for each $i=1,\cdots,d$, and we recover the usual dot-product of $\mathbb{K}^d$. 
\begin{remark}
The above construction works in particular when $\mathbb{K}$ is the finite field $\mathbb{F}_{p^n}$ with $p^n$ elements of characteristic $p$. In this case, any finite-dimensional vector space is actually finite, and we obtain a pairing between finite spaces (and therefore finite abelian groups). When $n=1$, we recover the construction of ``dual pairing vector space'' from~\cite{Okamoto1,Okamoto2}. 
\end{remark}

\noindent\textbf{Generalized duality of finite abelian groups: } Let $C_N$ be a cyclic group of order $N$, with generator $\gamma$. Let $A$ be any finite abelian group. A homomorphism of group $\chi\colon A\rightarrow C_N$ is called a \emph{character}. Since for every $a\in A$, $a^{|A|}=1_A$, it follows that $\chi(a)^{|A|}=1$.  Let $d$ be a divisor of $N$, and let $\chi\in\mathpzc{Hom}(\mathbb{Z}_d,C_N)$. Since for every $x$, $\chi(x\bmod d)^d=1$, it follows that $\mathsf{im}(\chi)$ is a subgroup of the unique cyclic subgroup $C_d$ of $C_N$ of order $d$. Therefore, $\chi(x\bmod d)=\gamma^{\frac{N}{d}j}$ for some $j=0,\cdots,d-1$ that depends on both $\chi$ and $x$. In particular, we have $\chi(1)=\gamma^{\frac{N}{d}i}$ for some $i$, and then, $\chi(x\bmod d)=\chi(x 1)=\chi(1)^x=\gamma^{\frac{N}{d}i x\bmod d}$. This means that all characters of $\mathbb{Z}_d$ have the form $\chi_i\colon \mathbb{Z}_d\rightarrow C_N$ with $\chi_i(x\bmod d)=\gamma^{\frac{N}{d}ix}$ for $i=1,\cdots,d$. It is not difficult to check that $\Psi\colon \mathbb{Z}_d\rightarrow \mathpzc{Hom}(\mathbb{Z}_d,C_N)$ given by $\Psi(i)=\chi_i$ is a group isomorphism. (Such a generalized approach for group characters has been used in~\cite{Poinsot} for other purposes.)

\begin{lemma}\label{cas-binaire}
Let $d_1,d_2$ be two divisors of $N$. Then, $\mathbb{Z}_{d_1}\times\mathbb{Z}_{d_2}$ and $\mathpzc{Hom}(\mathbb{Z}_{d_1}\times\mathbb{Z}_{d_2},C_N)$ are isomorphic. 
\end{lemma}
\begin{proof}
The proof is easy since it suffices to observe that $\mathpzc{Hom}(\mathbb{Z}_{d_1}\times\mathbb{Z}_{d_2},C_N)$ and $\mathpzc{Hom}(\mathbb{Z}_{d_1},C_N)\times \mathpzc{Hom}(\mathbb{Z}_{d_2},C_N)$ are isomorphic since we already know that $\mathpzc{Hom}(\mathbb{Z}_{d_i},C_N)$ is isomorphic to $\mathbb{Z}_{d_i}$ for $i=1,2$. Let $q_i\colon \mathbb{Z}_{d_i}\rightarrow \mathbb{Z}_{d_1}\times\mathbb{Z}_{d_2}$ be the canonical injection for $i=1,2$. Let us define the homomorphism of groups $\Phi\colon \mathpzc{Hom}(\mathbb{Z}_{d_1}\times\mathbb{Z}_{d_2},C_N)\rightarrow \mathpzc{Hom}(\mathbb{Z}_{d_1},C_N)\times \mathpzc{Hom}(\mathbb{Z}_{d_2},C_N)$ by $\Phi(\chi)=(\chi\circ i_1,\chi\circ i_2)$ which is obviously one-to-one. For $\chi^{(i)}\in \mathpzc{Hom}(\mathbb{Z}_{d_i},C_N)$, $i=1,2$, the map $\chi\colon (x_1,x_2)\mapsto \chi^{(1)}(x_1)\chi^{(2)}(x_2)$ belongs to  $ \mathpzc{Hom}(\mathbb{Z}_{d_1}\times\mathbb{Z}_{d_2},C_N)$, and $\Phi(\chi)=(\chi^{(1)},\chi^{(2)})$. \qed
\end{proof}
From lemma~\ref{cas-binaire} (and its proof), it is easy to see that $\mathpzc{Hom}(\bigoplus_{i=1}^m\mathbb{Z}_{d_i}^{m_i},C_N)\cong \bigoplus_{i=1}^{m}\mathpzc{Hom}(\mathbb{Z}_{d_i},C_N)^{m_i}\cong \bigoplus_{i=1}^m\mathbb{Z}_{d_i}^{m_i}$ for every divisor $d_i$ of $N$ and every integer $m_i$, $i=1,\cdots,m$. For each $i=1,\cdots,m$, and  $x=(x_1,\cdots,x_{m_i}), y=(y_1,\cdots,y_{m_i})\in \mathbb{Z}_{d_i}^{m_i}$, we define a \emph{dot-product} $$x\cdot y=\sum_{j=1}^{m_i}(x_j y_j)\bmod d_j\ .$$ Therefore, an isomorphism from $\bigoplus_{i=1}^m\mathbb{Z}_{d_i}^{m_i}$ to $\mathpzc{Hom}(\bigoplus_{i=1}^m\mathbb{Z}_{d_i}^{m_i},C_N)$ may be given by $\Psi(a^{(1)},\cdots,a^{(m)})=\chi_{a^{(1)},\cdots,a^{(m)}}$ for each $a^{(i)}\in \mathbb{Z}_{d_i}^{m_i}$, $i=1,\cdots,m$, where $$\chi_{a^{(1)},\cdots,a^{(m)}}(x^{(1)},\cdots,x^{(m)})=\prod_{i=1}^m\gamma^{\frac{N}{d_i}a^{(i)}\cdot x^{(i)}}$$ for every $(x^{(1)},\cdots,x^{(m)})\in \bigoplus_{i=1}^m\mathbb{Z}_{d_i}^{m_i}$ (so that $x^{(i)}\in\mathbb{Z}_{d_i}^{m_i}$ for each $i=1,\cdots,m$). Consequently, one obtains a bilinear map $\langle\cdot\mid\cdot\rangle\colon \bigoplus_{i=1}^m\mathbb{Z}_{d_i}^{m_i}\times \bigoplus_{i=1}^m\mathbb{Z}_{d_i}^{m_i}\rightarrow C_N$ such that $$\langle (x^{(1)},\cdots,x^{(m)})\mid (y^{(1)},\cdots,y^{(m)})\rangle=\prod_{i=1}^m\gamma^{\frac{N}{d_i}x^{(i)}\cdot y^{(i)}}$$ (where $x^{(i)},y^{(i)}\in \mathbb{Z}_{d_i}^{m_i}$, $i=1,\cdots,m$) which is right non-degenerate by construction. But this bilinear map is clearly symmetric, therefore it is actually non-degenerate and it defines a pairing. 
\begin{example}\label{exemple-de-lexemple}
Let $A=\bigoplus_{i=1}^m\mathbb{Z}_{d_i}^{m_i}$ for some integer $m$. 
\begin{enumerate}
\item Let $\gamma$ be a primitive element of the finite field $\mathbb{F}_{p^k}$ (see~\cite{Lidl}). Let us assume that $d_i$ is a divisor of $p^k-1$ for all $i=1,\cdots,m$. Then, we obtain a pairing from $A\times A$ to $\mathbb{F}_{p^k}^*$ given by 
$\langle (x^{(1)},\cdots,x^{(m)})\mid (y^{(1)},\cdots,y^{(m)})\rangle=\gamma^{\sum_{i=1}^{m}\frac{p^k-1}{d_i}(x^{(i)}\cdot y^{(i)})}$.
\item Let $\gamma=e^{\frac{2i\pi}{N}}$ be a primitive $N$-th square root of unity in the complex field. Let $d_i$ be a divisor of $N$ for each $i=1,\cdots,m$. Then, we obtain a pairing from $A\times A$ to $\mathbb{C}^*$ given by $\langle (x^{(1)},\cdots,x^{(m)})\mid (y^{(1)},\cdots,y^{(m)})\rangle=e^{\sum_{i=1}^{m}\frac{2i\pi}{d_i}(x^{(i)}\cdot y^{(i)})}$.
\end{enumerate}
\end{example}

\begin{remark}
The above construction still works when we consider usual group characters (see~\cite{Luong}) as it is shown in the second point of example~\ref{exemple-de-lexemple}. Let $A$ be an abelian group. In the classical setting a \emph{character} is a homomorphism of groups from $A$ to the multiplicative group $\mathbb{C}^*$. Torsion in $A$ implies that the image of a character belongs to the group of complex $\mathpzc{exp}(A)$-th roots of unity $C_{\mathpzc{exp}(A)}$. It is clear that for every decomposition of $A$ into a sum of cyclic groups $\displaystyle\bigoplus_{i=1}^m\mathbb{Z}_{d_i}^{m_i}$, $d_i$ divides $\mathpzc{exp}(A)$ (since $A$ contains an element of order $d_i$ for each $i$). The above machinery works. Moreover it may be recovered as follows in an abstract setting: let us denote by $\widehat{A}=\mathpzc{Hom}(A,C_{\mathpzc{exp}(A)})$ the group of characters, called \emph{dual group of $A$}. It is well-known that the double dual $\widehat{\widehat{A}}$ is naturally isomorphic to $A$. The natural bilinear map $\langle\cdot\mid\cdot\rangle\colon A\times \widehat{A}\rightarrow C_{\mathpzc{exp}(A)}$ is given by $\langle a\mid \chi\rangle=\chi(a)$. It is clearly right non-degenerate. Left non-degeneracy follows from $A\cong \widehat{\widehat{A}}$. Indeed, the isomorphism into consideration is given by $\widehat{a}(\chi)=\chi(a)$ for every $a\in A$, $\chi\in \widehat{A}$. Therefore, $\widehat{a}(\chi)=\chi(a)=1$ for every $\chi\in \widehat{A}$ implies that $\widehat{a}\equiv 1$ which is equivalent to $a=0_A$. According to lemma~\ref{degeneracy-of-can-implies-no-pairings}, this means that for every finite abelian group $A$, the canonical bilinear map $\otimes\colon A\times A\rightarrow A\otimes A$ is non-degenerate.
\end{remark}

% ---- Bibliography ----
%

\end{document}